\documentclass[10pt]{amsart}

\usepackage{amssymb}
\usepackage{amsthm}
\usepackage{amscd}
\usepackage{amsmath}
\usepackage{enumerate}
\usepackage{delarray}
\usepackage{hhline}

\title{On the quasi-derivation relation \endgraf for multiple zeta values}
\author{Tatsushi Tanaka}
\address{Graduate School of Mathematics, Kyushu University \endgraf
Fukuoka, 812-8581, Japan}
\email{t.tanaka@math.kyushu-u.ac.jp}

\theoremstyle{definition}
\newtheorem{thm}{Theorem}[section]
\newtheorem{defn}[thm]{Definition}

\newtheorem{rem}[thm]{Remark}
\newtheorem{cor}[thm]{Corollary}
\newtheorem{lem}[thm]{Lemma}
\newtheorem{prop}[thm]{Proposition}

\newtheorem{exmp}[thm]{Example}

\newtheorem{fact}[thm]{Fact}

\newtheorem*{ack}{Acknowledgement}

\newtheorem{Keyprop}[thm]{Key Proposition}

\DeclareFontEncoding{OT2}{}{}
\DeclareFontSubstitution{OT2}{cmr}{m}{n}

\DeclareFontFamily{OT2}{cmr}{\hyphenchar\font45 }
\DeclareFontShape{OT2}{cmr}{m}{n}{%
   <5><6><7><8><9>gen*wncyr%
   <10><10.95><12><14.4><17.28><20.74><24.88>wncyr10}{}
\DeclareFontShape{OT2}{cmr}{b}{n}{%
   <5><6><7><8><9>gen*wncyb%
   <10><10.95><12><14.4><17.28><20.74><24.88>wncyb10}{}

\DeclareMathAlphabet{\mathcyr}{OT2}{cmr}{m}{n}
\DeclareMathAlphabet{\mathcyb}{OT2}{cmr}{b}{n}
\SetMathAlphabet{\mathcyr}{bold}{OT2}{cmr}{b}{n}

\begin{document}

\maketitle

\begin{abstract}
Recently, Masanobu Kaneko introduced a conjecture on an extension of the derivation relation for multiple zeta values. The goal of the present paper is to present a proof of this conjecture by reducing it to a class of relations for multiple zeta values studied by Kawashima. In addition, some algebraic aspects of the quasi-derivation operator $\partial_n^{(c)}$ on $\mathbb{Q}\langle x,y \rangle$, which was defined by modeling a Hopf algebra developed by Connes and Moscovici, will be presented. 
\end{abstract}

\tableofcontents

\section{Introduction/Main Theorem}\label{sec1}
\noindent Let $n \ge 1$ be an integer. For each index set $(k_1,k_2,\ldots ,k_n)$ of positive integers with $k_1>1$, the multiple zeta value (MZV for short) is a real number defined by the convergent series 
$$ \displaystyle \zeta(k_1,k_2,\ldots ,k_n)=\sum_{m_1>m_2>\cdots >m_n>0}\frac{1}{m_1^{k_1}m_2^{k_2}\cdots m_n^{k_n}}. $$
We call the number $k_1+\cdots +k_n$ its weight and $n$ its depth. 

Throughout the present paper, we employ the algebraic setup introduced by Hoffman \cite{H} to study the quasi-derivation relation for MZV's. Let $\mathfrak{H}=\mathbb{Q}\langle x,y \rangle $ denote the non-commutative polynomial algebra over the rational numbers in two indeterminates $x$ and $y$, and let $\mathfrak{H}^1$ and $\mathfrak{H}^0$ denote the subalgebras $\mathbb{Q}+\mathfrak{H}y$ and $\mathbb{Q}+x \mathfrak{H}y$, respectively. The $\mathbb{Q}$-linear map $\mathit{Z}:\mathfrak{H}^0 \to \mathbb{R}$ is defined by $\mathit{Z}(1)=0$ and 
$$ \mathit{Z}(x^{k_1-1}yx^{k_2-1}y\cdots x^{k_n-1}y)=\zeta(k_1,k_2,\ldots ,k_n). $$
The degree (resp. degree with respect to $y$) of a word is the weight (resp. the depth) of the corresponding MZV. 

In Zagier's paper \cite{Z}, it is conjectured that the dimension of the $\mathbb{Q}$-vector space generated by MZV's of weight $k$ is $d_k$, the numbers determined by the recursion $d_0=1$, $d_1=0$, $d_2=1$, and $d_k=d_{k-2}+d_{k-3}$ for $k \ge 3$. Goncharov \cite{G} and Terasoma \cite{T} proved that the number $d_k$ gives an upper bound of the dimension of the space generated by MZV's of weight $k$. The number $d_k$ is far smaller than the total number $2^{k-2}$ of indices of weight $k$, hence there should be several relations among MZV's. In the present setup, finding a linear relation among MZV's corresponds to finding an element in $\mathrm{ker}\mathit{Z}\subset\mathfrak{H}^0$. 

Before stating the main theorem, the derivation relation for MZV's, which appeared in Ihara-Kaneko-Zagier \cite{IKZ}, is introduced. A derivation $\partial$ on $\mathfrak{H}$ is a $\mathbb{Q}$-linear endomorphism of $\mathfrak{H}$ satisfying the Leibnitz rule $\partial(ww^{\prime})=\partial(w)w^{\prime}+w\partial(w^{\prime})$. Such a derivation is uniquely determined by its images of generators $x$ and $y$. Let $z=x+y$. For each $n\ge 1$, the derivation $\partial_n:\mathfrak{H}\to\mathfrak{H}$ is defined by $\partial_n(x)=xz^{n-1}y$ and $\partial_n(y)=-xz^{n-1}y$. It follows immediately that $\partial_n(\mathfrak{H})\subset\mathfrak{H}^0$. 

\begin{fact}[Derivation Relation, \cite{IKZ}]\label{thm1}
{\it For any $n\ge 1$, we have $\partial_n(\mathfrak{H}^0)\subset \mathrm{ker}\mathit{Z}$.}
\end{fact}

The following extension of the operator $\partial_n$ is firstly defined by Kaneko \cite{K}. He modified the formula
$$ \displaystyle \partial_n=\frac{1}{(n-1)!}\mathrm{ad}(\theta)^{n-1}(\partial_1) $$
in \cite{IKZ}, where $\theta$ stands for the derivation on $\mathfrak{H}$ defined by $\theta(x)=\frac{1}{2}(xz+zx)$ and $\theta(y)=\frac{1}{2}(yz+zy)$, and $\mathrm{ad}(\theta)(\partial)=[\theta, \partial]:=\theta\partial -\partial\theta$. 

\begin{defn}\label{def1}
{\it Let $c$ be a rational number and $H$ the derivation on $\mathfrak{H}$ defined by $H(w)=\deg (w)w$ for any words $w \in \mathfrak{H}$. For each integer $n \ge 1$, the $\mathbb{Q}$-linear map $\partial_n^{(c)}:\mathfrak{H}\to\mathfrak{H}$ is defined by}
$$ \displaystyle \partial_n^{(c)}=\frac{1}{(n-1)!}\mathrm{ad}(\theta^{(c)})^{n-1}(\partial_1), $$
{\it where $\theta^{(c)}$ is the $\mathbb{Q}$-linear map defined by $\theta^{(c)}(x)=\theta (x)$, $\theta^{(c)}(y)=\theta (y)$ and the rule}
\begin{equation}
\theta^{(c)}(ww^{\prime})=\theta^{(c)}(w)w^{\prime}+w\theta^{(c)}(w^{\prime})+c\partial_1(w)H(w^{\prime}) \label{19}
\end{equation}
{\it for any $w,w^{\prime} \in \mathfrak{H}$.}
\end{defn}

If $c=0$, the quasi-derivation $\partial_n^{(c)}$ is reduced to the ordinary derivation $\partial_n$. If $c\neq 0$ and $n\ge 2$, the operator $\partial_n^{(c)}$ is no longer a derivation. Although the inclusion $\partial_n^{(c)}(\mathfrak{H})\subset\mathfrak{H}^0$ does not hold in general, we have $\partial_n^{(c)}(\mathfrak{H}^0)\subset\mathfrak{H}^0$ as will be shown in Proposition \ref{cor2}. Then, the main result of the present paper is stated.

\begin{thm}\label{mthm}
{\it For any $n\ge 1$ and any $c\in\mathbb{Q}$, we have $\partial_n^{(c)}(\mathfrak{H}^0)\subset \mathrm{ker}\mathit{Z}$.}
\end{thm}

When $c$ is viewed as a variable, $\partial_n^{(c)}(w)$ ($w\in\mathfrak{H}^0$) is a polynomial in $c$ of degree $n-1$. Then, Theorem \ref{mthm} implies that each coefficient with respect to $c$ of $\partial_n^{(c)}(w)$, $n\ge 1$, $w\in\mathfrak{H}^0$, is a relation among MZV's. We find that the derivation relation is the constant term of the quasi-derivation relation as a polynomial in $c$, and hence, the class of the derivation relation is contained in the class of the quasi-derivation relation and is again shown to be a class of relations among MZV's. 

The table below gives the maximal numbers of linearly independent relations supplied by each set of relations, together with the numbers $d_k$, the conjectural dimension of the space generated by MZV's of weight $k$, and $2^{k-2}$, the total number of indices of weight $k$. Computations were performed using Risa/Asir, an open source general computer algebra system. 

\begin{center}
\begin{tabular}{|c||c|c|c|c|c|c|c|c|c|c|c|c|}
\hline
weight $k$ & 3 & 4 & 5 & 6 & 7 & 8 & 9 & 10 & 11 & 12 & 13 & 14  \\
\hhline{|=#=|=|=|=|=|=|=|=|=|=|=|=|}
$2^{k-2}$ & 2 & 4 & 8 & 16 & 32 & 64 & 128 & 256 & 512 & 1024 & 2048 & 4096  \\
\hline
$d_k$ & 1 & 1 & 2 & 2 & 3 & 4 & 5 & 7 & 9 & 12 & 16 & 21  \\
\hline
$\partial_n$ & 1 & 2 & 5 & 10 & 22 & 44 & 90 & 181 & 363 & 727 & 1456 & 2912  \\
\hline
$\partial_n^{(c)}$ & 1 & 2 & 5 & 10 & 23 & 46 & 98 & 200 & 410 & 830 & 1679 & $\cdots$  \\
\hline
Ohno & 1 & 2 & 5 & 10 & 23 & 46 & 98 & 199 & 411 & 830 & 1691 & $\cdots$  \\
\hline
lin. K. & 1 & 2 & 5 & 10 & 23 & 46 & 98 & 200 & 413 & 838 & 1713 & $\cdots$   \\
\hline
alg. K. & 1 & 3 & 6 & 14 & 29 & 60 & 123 & 249 & 503 & 1012 & $\cdots$ & $\cdots$  \\
\hline
\end{tabular}
\end{center}

\vspace{5pt}

The label `$\partial_n$' denotes the class of the derivation relation and the label `$\partial_n^{(c)}$' the class of the quasi-derivation relation, generated by the coefficients of $\partial_n^{(c)}(w)$, $w\in\mathfrak{H}^0$, as a polynomial in $c$. The label `Ohno' denotes Ohno's relation; see \cite{O} for details. Kawashima proved in \cite{Kawa} a class of algebraic relations among MZV's. The label `lin. K.' denotes the linear part of Kawashima's relation and the bottom column `alg. K.' the union of the linear part and the degree 2 part of Kawashima's relation, where products of MZV's are linearly expanded according to the (iterated integral) shuffle product rule; see \cite{IKZ} for example. The sequence of alg.K. suggests that the whole set of Kawashima's relation is enough to reduce the dimensions of the space generated by MZV's to the conjectural ones. 

Further experiments using Risa/Asir enable us to find some facts or expectations. In fact, the sequence of `lin. K.' appears again as the sequence of following three classes, `$\partial_n^{(c)}$' $\cup$ `Ohno', `$\partial_n^{(c)}$' $\cup$ `lin. K.', `Ohno' $\cup$ `lin. K.', up to weight $13$. Hence, three classes `$\partial_n^{(c)}$', `Ohno' and `lin. K.' coincide up to weight $9$. In addition, from weight $10$ to $13$ (and probably for higher weights), `$\partial_n^{(c)}$' and `Ohno' are different classes but both are contained in `lin. K.' properly. 

Ohno's relation is known to be equivalent to the union of two classes of relations, the derivation relation and the duality formula. A proof of this equivalence is given in Appendix 2. (See \cite{AK}, too.) Kawashima showed in \cite{Kawa} that the duality formula is contained in `lin. K.' The quasi-derivation relation is also contained in `lin. K.', which is shown in the present paper. Although the table and further experiments stated above imply that these two classes are equivalent, only one side inclusion: `$\partial_n^{(c)}$' $\cup$ $\{$the duality formula$\}$ $\subset$ `lin. K.' can be shown herein.

\ack
The author is grateful to Professor Masanobu Kaneko and the referee for many useful comments and advice. He also thanks Professor Masayuki Noro for his help in making programs by using Risa/Asir.

\section{Proof of Main Result}\label{sec2}

\noindent The main theorem (Theorem \ref{mthm}) is proven by reducing the theorem to the following Kawashima's relation. 

Let $z_k=x^{k-1}y$ for $k\ge 1$. The harmonic product $\ast :\mathfrak{H}^1\times\mathfrak{H}^1\to\mathfrak{H}^1$ is a $\mathbb{Q}$-bilinear map defined by the following rules. 
\begin{eqnarray*}
&\mathrm{i})& \quad\text{For any}~w\in\mathfrak{H}^1,~1\ast w=w\ast 1=w. \\
&\mathrm{ii})&  \quad\text{For any}~w, w^{\prime}\in\mathfrak{H}^1~\text{and any}~k,l\ge 1, \\
&{}& \quad z_kw\ast z_lw^{\prime}=z_k(w\ast z_lw^{\prime})+z_l(z_kw\ast w^{\prime})+z_{k+l}(w\ast w^{\prime}). 
\end{eqnarray*}
This is, as shown in \cite{H}, an associative and commutative product on $\mathfrak{H}^1$. 

Denote by $\varepsilon$ the automorphism of $\mathfrak{H}$ defined by $\varepsilon(x)=z=x+y$ and $\varepsilon(y)=-y$. For any $w\in\mathfrak{H}$, define the operator $L_w$ on $\mathfrak{H}$ by $L_w(w^{\prime})=ww^{\prime}~(w^{\prime}\in\mathfrak{H})$. Next, the linear part of Kawashima's relation \cite[Corollary~4.9]{Kawa} is stated using the notation of the present paper. 

\begin{fact}[Kawashima's Relation]\label{thm2}
{\it $L_x\varepsilon(\mathfrak{H}y\ast\mathfrak{H}y)\subset \mathrm{ker}\mathit{Z}$.}
\end{fact}

Let $\tau$ be the anti-automorphism of $\mathfrak{H}$ defined by $\tau(x)=y$ and $\tau(y)=x$. The duality formula states that $(1-\tau)(\mathfrak{H}^0)\subset \mathrm{ker}\mathit{Z}$. To prove Theorem \ref{mthm}, the inclusion 
\begin{equation}
\partial_n^{(c)}(\mathfrak{H}^0)\subset\tau L_x\varepsilon (\mathfrak{H}y\ast\mathfrak{H}y) \label{1}
\end{equation}
is shown. In \cite{Kawa}, Kawashima proved that Kawashima's relation contains the duality formula:
\begin{displaymath}
(1-\tau)(\mathfrak{H}^0)\subset L_x\varepsilon(\mathfrak{H}y\ast\mathfrak{H}y),
\end{displaymath}
and hence,
\begin{displaymath}
\mathrm{RHS}~\mathrm{of}~(\ref{1})=(1-(1-\tau)) L_x\varepsilon (\mathfrak{H}y\ast\mathfrak{H}y)\subset L_x\varepsilon (\mathfrak{H}y\ast\mathfrak{H}y).
\end{displaymath}
Therefore, based on Kawashima's relation, the inclusion (\ref{1}) gives Theorem \ref{mthm}. 

To prove (\ref{1}), the following key identity, which involves several operators, is established. For any $w\in\mathfrak{H}$, let $R_w$ be the operator defined by $R_w(w^{\prime})=w^{\prime}w~(w^{\prime}\in\mathfrak{H})$. The operator $\mathcal{H}_w$ on $\mathfrak{H}^1$ for any $w\in\mathfrak{H}^1$ given by $\mathcal{H}_w(w^{\prime})=w\ast w^{\prime}~(w^{\prime}\in\mathfrak{H}^1)$ is also introduced. Set $\chi_x=\tau L_x\varepsilon$.

\begin{Keyprop}\label{keyprop}
{\it For any $n\ge 1$ and any $c\in\mathbb{Q}$, there exists an element $w=w(n,c)\in\mathfrak{H}y$ such that $\partial_n^{(c)}\chi_x = \chi_x\mathcal{H}_w$ on $\mathfrak{H}^1$. In other words, the following commutative diagram holds:
\[
\begin{CD}
\mathfrak{H}^1 @>{\mathcal{H}_w}>>    \mathfrak{H}^1 \\
@V{\chi_x}VV                          @VV{\chi_x}V   \\
\mathfrak{H}^0 @>>{\partial_n^{(c)}}> \mathfrak{H}^0 
\end{CD}
\]
}
\end{Keyprop}

This proposition implies the explicit expression of $w=w(n,c)$. The identity holds on $\mathfrak{H}^1$, and hence on $\mathbb{Q}$. Applying the operetors to $1$ $(\in\mathbb{Q})$, we find $\partial_n^{(c)}\chi_x(1)=\chi_x\mathcal{H}_w(1)=\chi_x(w)$, and hence, we have $w=\chi_x^{-1}\partial_n^{(c)}(y)=\varepsilon L_x^{-1}\tau\partial_n^{(c)}(y)$. Because of Proposition \ref{cor2}, the operator $L_x^{-1}$ makes sence (, that means to remove the head letter $x$ of every term). 

The proof of this proposition is the technical core of the present paper and will be carried out in the next two sections. In addition, various beneficial properties of operators, including the commutativity of $\partial_n^{(c)}$'s, are proven. 

Assuming Key Proposition \ref{keyprop}, the proof of Theorem \ref{mthm} proceeds as follows. First, note that it is sufficient to prove the inclusion $\partial_n^{(c)}(x\mathfrak{H}y)\subset\tau L_x\varepsilon (\mathfrak{H}y\ast\mathfrak{H}y)$ instead of (\ref{1}) because $\mathfrak{H}^0=\mathbb{Q}+x\mathfrak{H}y$ and $\partial_n^{(c)}(\mathbb{Q})=\{0\}$. Take any $w_0\in x\mathfrak{H}y$. Since $\varepsilon$ is an automorphism of $\mathfrak{H}$ and $\varepsilon (y)=-y$, we have 
$ R_y\tau\varepsilon (\mathfrak{H}y)=x\mathfrak{H}y, $
and hence, there is an element $w_1\in\mathfrak{H}y$ such that 
$ w_0=R_y\tau\varepsilon (w_1). $
By Key Proposition \ref{keyprop}, there exists $w_2\in\mathfrak{H}y$ satisfying 
$ \partial_n^{(c)}R_y\tau\varepsilon = \tau L_x\varepsilon\mathcal{H}_{w_2}. $
Therefore, we have 
$$ \partial_n^{(c)}(w_0)=\partial_n^{(c)}R_y\tau\varepsilon (w_1)=\tau L_x\varepsilon\mathcal{H}_{w_2}(w_1)=\tau L_x\varepsilon (w_2\ast w_1). $$
This proves $(\ref{1})$ and Theorem \ref{mthm} is established.

\section{Commutativity of $\partial_n^{(c)}$}\label{sec4}

\noindent To prove Key Proposition \ref{keyprop}, several properties of various operators are needed, and the commutativity of $\partial_n^{(c)}$'s must first be proven. 

\begin{prop}\label{prop5}
{\it Let $c\in\mathbb{Q}$. For any $n,m\ge 1$, we have $[\partial_n^{(c)},\partial_m^{(c)}]=0$. }
\end{prop}

As mentioned earlier, the operator $\partial_n^{(c)}$ is no longer a derivation if $c\neq 0$ and $n\ge 2$ and does not satisfy the Leibniz rule, instead, satisfying the rules such as
\begin{eqnarray*}
\displaystyle &\partial_2^{(c)}(ww^{\prime})& = \partial_2^{(c)}(w)w^{\prime}+w\partial_2^{(c)}(w^{\prime})+c\partial_1^{(c)}(w)\partial_1^{(c)}(w^{\prime}), \\
\displaystyle &\partial_3^{(c)}(ww^{\prime})& = \partial_3^{(c)}(w)w^{\prime}+w\partial_3^{(c)}(w^{\prime})+\frac{1}{2}c\partial_2^{(c)}(w)\partial_1^{(c)}(w^{\prime})+\frac{3}{2}c\partial_1^{(c)}(w)\partial_2^{(c)}(w^{\prime}) \\
&{}& \displaystyle \quad +\frac{1}{2}c^2{\partial_1^{(c)}}^2(w)\partial_1^{(c)}(w^{\prime}), 
\end{eqnarray*}
for any $w,w^{\prime}\in\mathfrak{H}$, which can be checked using the definition of the operator $\partial_n^{(c)}$ and Proposition \ref{prop5}, namely, the commutativity of $\partial_n^{(c)}$. The subalgebra $\mathsf{A}^{(c)}$ of linear endomorphisms of $\mathfrak{H}$ generated by $\partial_1$, $\theta^{(c)}$ and $H$ (, and hence, $\partial_n^{(c)}\in\mathsf{A}^{(c)}$) has the structure of Connes-Moscovici's Hopf algebra (see \cite{CM}), which is helpful to calculate such rule of $\partial_n^{(c)}$. 

To prove Proposition \ref{prop5}, the following several operators are needed. Recall the left and right multiplication operators are both additive as well as multiplicative ($L_{ww^{\prime}}=L_wL_{w^{\prime}}$) and anti-multiplicative ($R_{ww^{\prime}}=R_{w^{\prime}}R_w$), respectively. 

\begin{defn}\label{def2}
{\it Let $c$ be a rational number. The operators $\{\phi_n^{(c)}\}_{n=0}^{\infty}$ are defined by $\phi_0^{(c)}=\mathrm{id}_{\mathfrak{H}}$ and the recursive rule: }
\begin{equation}
\displaystyle \phi_n^{(c)}=\frac{1}{n}\bigl([\theta^{(c)},\phi_{n-1}^{(c)}]+\frac{1}{2}(R_z\phi_{n-1}^{(c)}+\phi_{n-1}^{(c)}R_z)+c\partial_1\phi_{n-1}^{(c)}\bigl), \label{4}
\end{equation}
{\it for $n \ge 1$.}
\end{defn}

\begin{lem}\label{lem1}
{\it For $n\ge 1$, let $\psi_n^{(c)}=R_y\phi_{n-1}^{(c)}R_x$. The operators $\{\psi_n^{(c)}\}_{n=1}^{\infty}$ satisfy $\psi_1^{(c)}=R_{xy}$ and the recursive rule
$$ \displaystyle \psi_n^{(c)}=\frac{1}{n-1}\bigl([\theta^{(c)},\psi_{n-1}^{(c)}]-\frac{1}{2}(R_z\psi_{n-1}^{(c)}+\psi_{n-1}^{(c)}R_z)-c\psi_{n-1}^{(c)}\partial_1\bigl) $$
for $n \ge 2$.}
\end{lem}

\begin{proof}
The lemma is proven by induction on $n$. The lemma holds for $n=1$ because $R_{xy}=R_yR_x$. Assume that the lemma is proved for $n$. Because of the identities $[\theta^{(c)}, R_u]=R_{\theta (u)}+cR_u\partial_1=\frac{1}{2}(R_zR_u+R_uR_z)+cR_u\partial_1$ for $u=x$ or $y$, the recursive rule of $\phi_n^{(c)}$ and the induction hypothesis, we have
\begin{eqnarray*}
\displaystyle [\theta^{(c)},\psi_{n-1}^{(c)}] &=& [\theta^{(c)},R_y\phi_{n-2}^{(c)}R_x] \\
\displaystyle &=& R_y\phi_{n-2}^{(c)}[\theta^{(c)},R_x]+R_y[\theta^{(c)},\phi_{n-2}^{(c)}]R_x+[\theta^{(c)},R_y]\phi_{n-2}^{(c)}R_x \\
\displaystyle &=& (n-1)\psi_n^{(c)}+\frac{1}{2}(R_z\psi_{n-1}^{(c)}+\psi_{n-1}^{(c)}R_z)+c\psi_{n-1}^{(c)}\partial_1.
\end{eqnarray*}
Therefore, the lemma is proven.
\end{proof}

In order to prove Proposition \ref{prop5}, the following general property of a $\mathbb{Q}$-linear map on $\mathfrak{H}$ is needed.

\begin{lem}\label{lem9}
{\it A $\mathbb{Q}$-linear map $f:\mathfrak{H}\to\mathfrak{H}$ satisfying $[f,R_x]=[f,R_y]=0$ and $f(1)=0$ is necessarily a zero map. }
\end{lem}

\begin{proof}
Since $f$ is $\mathbb{Q}$-linear, it is only necessary to show $f(w)=0$ for any words $w\in\mathfrak{H}$. Write $w=u_1u_2\cdots u_n$ with $u_1,u_2,\ldots ,u_n\in\{x,y\}$. Since $[f,R_{u_i}]=0$ for any $1\le i\le n$ by assumption, we have
$$ f(w)=f(u_1u_2\cdots u_n)=f(u_1u_2\cdots u_{n-1})u_n=\cdots =f(1)u_1u_2\cdots u_n=0. $$
\end{proof}

Next, the commutativity property of $\partial_n^{(c)}$ is given. Instead of Proposition \ref{prop5}, the following slightly general statement is shown.

\begin{prop}\label{comm}
{\it For any $n,m\ge 1$ and any $c,c^{\prime}\in\mathbb{Q}$, we have $[\partial_n^{(c)},\partial_m^{(c^{\prime})}]=0$. }
\end{prop}

\begin{proof} 
In the following, $(\mathrm{A}_n)$ and $(\mathrm{B}_n)$ are shown inductively as $(\mathrm{A}_1),(\mathrm{B}_1)\Rightarrow(\mathrm{A}_2)\Rightarrow(\mathrm{B}_2)\Rightarrow(\mathrm{A}_3)\Rightarrow(\mathrm{B}_3)\Rightarrow(\mathrm{A}_4)\Rightarrow\cdots $. 

Let $\mathrm{sgn}(x)=1$ and $\mathrm{sgn}(y)=-1$. 
\begin{eqnarray*}
&(\mathrm{A}_n)&\quad [\partial_n^{(c)},R_u]=\mathrm{sgn}(u)\psi_n^{(c)}~\mathrm{for}~\mathrm{any}~c \in \mathbb{Q}~\mathrm{and}~\mathrm{any}~u\in\{x,y\}. \\
&(\mathrm{B}_n)&\quad [\partial_n^{(c)},\partial_i^{(c^{\prime})}]=0~\mathrm{for}~\mathrm{any}~1 \le i \le n~\mathrm{and}~\mathrm{any}~c,c^{\prime} \in \mathbb{Q}.
\end{eqnarray*}
Note that if $(\mathrm{B}_n)$'s for any $n\ge 1$ can be shown, the proposition is shown. 

Note the following three considerations. First, note that the statement $(\mathrm{A}_n)$ means that, for any $w\in\mathfrak{H}$ and any $u\in\{x,y\}$, 
$$ \partial_n^{(c)}(wu)=\partial_n^{(c)}(w)u+\mathrm{sgn}(u)\psi_n^{(c)}(w) $$
and implies 
\begin{eqnarray*}
(\alpha_n)&\quad [\partial_n^{(c)},R_z]=0~\mathrm{for}~\mathrm{any}~c \in \mathbb{Q} 
\end{eqnarray*}
where $z=x+y$. 

Second, let 
\begin{eqnarray*}
&(\mathrm{B}_{n,i})&\quad [\partial_n^{(c)},\partial_i^{(c^{\prime})}]=0~\mathrm{for}~\mathrm{a}~\mathrm{fixed}~1 \le i \le n~\mathrm{and}~\mathrm{any}~c,c^{\prime} \in \mathbb{Q}. 
\end{eqnarray*}
Clearly, the statement $(\mathrm{B}_n)$ is equivalent to the union of $(\mathrm{B}_{n,i})$'s for $1\le i\le n$. Owing to Lemma \ref{lem9} and $[\partial_n^{(c)},\partial_i^{(c^{\prime})}](1)=0$ by $\partial_n^{(c)}(\mathbb{Q})=0$, each $(\mathrm{B}_{n,i})$ is equivalent to the statement
\begin{eqnarray*}
&(\mathrm{B}_{n,i}^{\prime})&\quad [[\partial_n^{(c)},\partial_i^{(c^{\prime})}],R_u]=0 \\
&{}& \quad \mathrm{for}~\mathrm{a}~\mathrm{fixed}~1 \le i \le n,~\mathrm{any}~c,c^{\prime} \in \mathbb{Q},~\mathrm{and}~\mathrm{any}~u\in\{x,y\}. 
\end{eqnarray*}
Instead of $(\mathrm{B}_{n+1})$, $(\mathrm{B}_{n+1,i}^{\prime})$'s for $1\le i\le n+1$ are shown by induction on $i$. 

Third, note that the commutative polynomial ring $\mathbb{Q}[R_z,\partial_1^{(c)},\ldots ,\partial_n^{(c)}]$ can be considered if $(\mathrm{A}_i)$ (hence $(\alpha_i)$) and $(\mathrm{B}_i)$ hold for all $1\le i\le n$. Let $\mathbb{Q}[R_z,\partial_1^{(c)},\ldots ,\partial_n^{(c)}]_{(i)}$ denote the degree $i$ homogenous part with $\deg(R_z)=1$ and $\deg(\partial_d^{(c)})=d$. These assumptions together with the recursive rule (\ref{4}) give us the fact
\begin{eqnarray*}
&(\beta_n)&\quad \phi_n^{(c)} \in \mathbb{Q}[R_z,\partial_1^{(c)},\ldots ,\partial_n^{(c)}]_{(n)}~\mathrm{for}~\mathrm{any}~c \in \mathbb{Q}.
\end{eqnarray*}
Based on the above considerations, the proof of $(\mathrm{A}_n)$ and $(\mathrm{B}_n)$ is now given. Since $[\partial_1^{(c)},R_u](w)=\partial_1^{(c)}(wu)-\partial_1^{(c)}(w)u=w\partial_1^{(c)}(u)=R_{\partial_1^{(c)}(u)}(w)$ for $w\in\mathfrak{H}$ and $\mathrm{sgn}(u)\psi_1^{(c)}=\mathrm{sgn}(u)R_{xy}=R_{\partial_1^{(c)}(u)}$ for any $u\in\{x,y\}$, the statement $(\mathrm{A}_1)$ holds. The statement $(\mathrm{B}_1)$ is trivial because $\partial_1^{(c)}=\partial_1^{(c^{\prime})}=\partial_1$ for any $c,c^{\prime}\in\mathbb{Q}$.

Assume that $(\mathrm{A}_n)$ (hence $(\alpha_n)$) and $(\mathrm{B}_n)$ are proven. By the definition of $\partial_{n+1}^{(c)}$, 
$$ \displaystyle n[\partial_{n+1}^{(c)},R_u] = [[\theta^{(c)},\partial_n^{(c)}],R_u]. $$
Using Jacobi's identity, the right-hand side equals
$$ \displaystyle -[[\partial_n^{(c)},R_u],\theta^{(c)}]-[[R_u,\theta^{(c)}],\partial_n^{(c)}]. $$
By $(\mathrm{A}_n)$ and $[\theta^{(c)}, R_u]=R_{\theta (u)}+cR_u\partial_1$ for $u\in\{x,y\}$, this yields
$$ \displaystyle -\mathrm{sgn}(u)[\psi_n^{(c)},\theta^{(c)}]+[R_{\theta (u)}+cR_u\partial_1,\partial_n^{(c)}]. $$
Using $R_{\theta (u)}=\frac{1}{2}(R_zR_u+R_uR_z)$, $(\alpha_n)$, and $(\mathrm{B}_n)$, 
$$ \displaystyle [R_{\theta (u)}+cR_u\partial_1,\partial_n^{(c)}] = \frac{1}{2}(R_z[R_u,\partial_n^{(c)}]+[R_u,\partial_n^{(c)}]R_z)+c[R_u,\partial_n^{(c)}]\partial_1. $$
Hence, using $(\mathrm{A}_n)$, we have 
$$ [\partial_{n+1}^{(c)},R_u] = \frac{\mathrm{sgn}(u)}n\bigl([\theta^{(c)},\psi_n^{(c)}]-\frac{1}{2}(R_z\psi_n^{(c)}+\psi_n^{(c)}R_z)-c\psi_n^{(c)}\partial_1\bigl)=\mathrm{sgn}(u)\psi_{n+1}^{(c)}, $$
and therefore $(\mathrm{A}_{n+1})$ (as well as $(\alpha_{n+1})$) is proven. 

In order to prove $(\mathrm{B}_{n+1})$, assume that all $(\mathrm{A}_{j})$'s (hence $(\alpha_{j})$'s) for $1\le j\le n+1$ and all $(\mathrm{B}_{j})$'s (hence $(\beta_{j})$'s) for $1\le j\le n$ are proven. As mentioned above, $(\mathrm{B}_{n+1,i}^{\prime})$'s for $1\le i\le n+1$ are proven instead of $(\mathrm{B}_{n+1})$. Using Jacobi's identity, we have
\begin{equation}
[[\partial_{n+1}^{(c)},\partial_i^{(c^{\prime})}],R_u] = -[[\partial_i^{(c^{\prime})},R_u],\partial_{n+1}^{(c)}]-[[R_u,\partial_{n+1}^{(c)}],\partial_i^{(c^{\prime})}] \label{6}
\end{equation}
for every $1\le i\le n+1$. By $(\mathrm{A}_{i})$ and Lemma \ref{lem1}, 
$$ [\partial_i^{(c)},R_u]=\mathrm{sgn}(u)\psi_i^{(c)}=\mathrm{sgn}(u)R_y\phi_{i-1}^{(c)}R_x $$ 
for any $1\le i\le n+1$, any $c\in\mathbb{Q}$, and any $u\in\{x,y\}$, and hence,
$$ -\mathrm{sgn}(u)(\mathrm{RHS}~\mathrm{of}~(\ref{6}))=[R_y\phi_{i-1}^{(c^{\prime})}R_x,\partial_{n+1}^{(c)}]-[R_y\phi_n^{(c)}R_x,\partial_i^{(c^{\prime})}]. $$
The right-hand side is equal to the sum
\begin{equation}
\left.
\begin{array}{ll}
R_y\phi_{i-1}^{(c^{\prime})}[R_x,\partial_{n+1}^{(c)}]+R_y[\phi_{i-1}^{(c^{\prime})},\partial_{n+1}^{(c)}]R_x+[R_y,\partial_{n+1}^{(c)}]\phi_{i-1}^{(c^{\prime})}R_x\qquad\qquad \\
\qquad\qquad -R_y\phi_n^{(c)}[R_x,\partial_i^{(c^{\prime})}]-R_y[\phi_n^{(c)},\partial_i^{(c^{\prime})}]R_x-[R_y,\partial_i^{(c^{\prime})}]\phi_n^{(c)}R_x.
\end{array}
\right. \label{2}
\end{equation}
If $i=1$, we have $\phi_{i-1}^{(c^{\prime})}=\phi_0^{(c^{\prime})}=\mathrm{id}_{\mathfrak{H}}$, and hence, 
$$ [\phi_{i-1}^{(c^{\prime})},\partial_{n+1}^{(c)}]=[\phi_0^{(c^{\prime})},\partial_{n+1}^{(c)}]=0. $$ 
Thanks to $(\beta_n)$ and the identity $\partial_1^{(c^{\prime})}=\partial_1^{(c)}(=\partial_1)$, we also have
$$ [\phi_n^{(c)},\partial_i^{(c^{\prime})}]=[\phi_n^{(c)},\partial_1^{(c^{\prime})}]=0. $$ 
Thus, in this case, the entire expression $(\ref{2})$ turns into
\begin{equation}
-R_y\psi_{n+1}^{(c)}+\psi_{n+1}^{(c)}R_x+R_y\phi_{n}^{(c)}\psi_1^{(c^{\prime})}-\psi_1^{(c^{\prime})}\phi_{n}^{(c)}R_x \label{7}
\end{equation}
by $(\mathrm{A}_1)$ and $(\mathrm{A}_{n+1})$. Using Lemma \ref{lem1}, $\psi_1^{(c^{\prime})}=R_yR_x$ and $R_z=R_x+R_y$, we obtain the expression (\ref{7}) equals $-R_yR_z\phi_n^{(c)}R_x+R_y\phi_n^{(c)}R_zR_x$. The right-hand side becomes zero because $[R_z,\phi_n^{(c)}]=0$ by $(\beta_n)$. Thus, $(\mathrm{B}_{n+1,1}^{\prime})$ (as well as $(\mathrm{B}_{n+1,1})$) is proven. \\
In order to conclude the expression (\ref{2}) equals zero for $i$ with $1<i\le n+1$, assume that $(\mathrm{B}_{n+1,i-1})$ (hence $(\mathrm{B}_{n+1,i-1}^{\prime})$) is proven. We then obtain 
$$ [\phi_{i-1}^{(c^{\prime})},\partial_{n+1}^{(c)}]=0 $$ 
based on $(\beta_{i-1})$, $(B_{n+1,i-1})$, and $(\alpha_{n+1})$. In addition, we obtain
$$ [\phi_n^{(c)},\partial_i^{(c^{\prime})}]=0~\mathrm{for}~1<i\le n+1 $$ 
by $(\beta_n)$, $(\mathrm{B}_n)$, and $(\alpha_i)$ (when $1<i<n+1$) or by $(\beta_n)$, $(\mathrm{B}_{n+1,n})$, and $(\alpha_{n+1})$ (when $i=n+1$). Thus, in this case, the entire expression $(\ref{2})$ turns into
\begin{equation}
-R_y\phi_{i-1}^{(c^{\prime})}\psi_{n+1}^{(c)}+\psi_{n+1}^{(c)}\phi_{i-1}^{(c^{\prime})}R_x+R_y\phi_{n}^{(c)}\psi_{i}^{(c^{\prime})}-\psi_{i}^{(c^{\prime})}\phi_{n}^{(c)}R_x \label{8}
\end{equation}
by $(\mathrm{A}_i)$ and $(\mathrm{A}_{n+1})$. Using Lemma \ref{lem1} and $R_z=R_x+R_y$, we obtain the expression (\ref{8}) equals $-R_y\phi_{i-1}^{(c^{\prime})}R_z\phi_n^{(c)}R_x+R_y\phi_n^{(c)}R_z\phi_{i-1}^{(c^{\prime})}R_x$. The right-hand side becomes zero because the operators $\phi_{i-1}^{(c^{\prime})}$, $\phi_n^{(c)}$ and $R_z$ commute with one another. Thus, $(\mathrm{B}_{n+1,i}^{\prime})$ (as well as $(\mathrm{B}_{n+1,i})$) holds, and by induction, we obtain $(\mathrm{B}_{n+1})$. This concludes the proof of the proposition.
\end{proof}


According to $(\beta_n)$, $\phi_n^{(c)}$ commutes with $R_z$, and so the recursive rule $(\ref{4})$ is simplified as
\begin{equation}
\displaystyle \phi_n^{(c)}=\frac{1}{n}\bigl([\theta^{(c)},\phi_{n-1}^{(c)}]+(R_z+c\partial_1)\phi_{n-1}^{(c)}\bigl). \label{11}
\end{equation}
Masanobu Kaneko pointed out a formula for $\phi_n^{(c)}$, 
$$ \displaystyle R_z\phi_n^{(c)}=\frac{1}{n!}\mathrm{ad}(\theta^{(c)})^n(R_z). $$
This is shown by using $[\theta^{(c)},R_z]=R_{\theta (z)}+cR_z\partial_1$ and the recursive formula $(\ref{11})$. 

Using Proposition \ref{comm}, we also obtain

\begin{prop}\label{cor2}
{\it We have $\partial_n^{(c)}(\mathbb{Q}\cdot x+\mathbb{Q}\cdot y+\mathfrak{H}^0) \subset \mathfrak{H}^0$ for any integer $n\ge 1$ and any $c \in \mathbb{Q}$. }
\end{prop}
\begin{proof}
By Lemma \ref{lem1} and $(\mathrm{A}_n)$ in the proof of Proposition \ref{comm}, we have
\begin{equation}
\partial_n^{(c)}(wu)=\partial_n^{(c)}(w)u+\mathrm{sgn}(u)\phi_{n-1}^{(c)}(wx)y\quad(w\in\mathfrak{H},u\in\{x,y\}). \label{13}
\end{equation}
This implies
\begin{equation}
\partial_n^{(c)}(\mathbb{Q}\cdot x+\mathfrak{H}^1) \subset \mathfrak{H}^1. \label{17}
\end{equation}
Next, the proposition is shown by induction on $n$. The proposition holds for $n=1$ because $\partial_1^{(c)}=\partial_1$. Assume that the proposition is proven for $n-1$. Using equation (\ref{13}), $(\beta_{n-1})$, and $\partial_n^{(c)}(1)=0$, by induction on the degree of a word, we find that both $\partial_n^{(c)}(x)$ and $\partial_n^{(c)}(xwy)$ for any words $w\in\mathfrak{H}$ begin with the letter $x$ (hence, using (\ref{17}), $\partial_n^{(c)}(x),\partial_n^{(c)}(xwy)\in\mathfrak{H}^0$). In addition, because of ($\alpha_n$), we have $\partial_n^{(c)}(z)=\partial_n^{(c)}R_z(1)=R_z\partial_n^{(c)}(1)=0$ where $z=x+y$, and hence, we have $\partial_n^{(c)}(y)=-\partial_n^{(c)}(x)\in\mathfrak{H}^0$. Therefore, the proposition is proven for $n$.
\end{proof}

\section{Proof of Key Proposition}\label{sec3}

\noindent In this section, the proof of Key Proposition \ref{keyprop} is given. 

Denote by $\mathfrak{H}_n^1$ the weight $n$ homogenous part of $\mathfrak{H}^1$. Recall that $z_k=x^{k-1}y$ for $k\ge 1$ as defined in $\S \ref{sec2}$. Let $\mathfrak{W}$ be the $\mathbb{Q}$-vector space generated by $\{\mathcal{H}_w|w\in\mathfrak{H}^1\}$, and $\mathfrak{W}_n$ the vector subspace of $\mathfrak{W}$ generated by $\{\mathcal{H}_w|w\in\mathfrak{H}^1_n\}$. Let $\mathfrak{W}^{\prime}$ be the $\mathbb{Q}$-vector space generated by $\{L_{z_k}\mathcal{H}_w|k\ge 1,~w\in\mathfrak{H}^1\}$, and $\mathfrak{W}^{\prime}_n$ the vector subspace of $\mathfrak{W}^{\prime}$ generated by $\{L_{z_k}\mathcal{H}_w|1\le k\le n,~w\in\mathfrak{H}^1_{n-k}\}$. The $\mathbb{Q}$-linear map $ \lambda : \mathfrak{W}^{\prime}\to\mathfrak{W}$ is defined by $\lambda(L_{z_k}\mathcal{H}_w)=\mathcal{H}_{z_kw}$. 

\begin{rem}
Here, we show the well-definedness of the map $\lambda$. Assume that
\begin{equation}
\displaystyle \sum_{(z_k,w)}C_{(z_k,w)}L_{z_k}\mathcal{H}_w=0~~(\in\mathfrak{W}), \label{20}
\end{equation}
where the sum is over different pairs of words $(z_k,w)$. Applying $(\ref{20})$ to $1\in\mathfrak{H}$, we have
$$ \displaystyle \sum_{(z_k,w)}C_{(z_k,w)}z_kw=0. $$
Then, for each $z_k$, we have
$$ \displaystyle \sum_wC_{(z_k,w)}w=0 $$
where the sum is over different words $w$. Therefore, each coefficient $C_{(z_k,w)}$ becomes zero, and hence, $L_{z_k}\mathcal{H}_w$'s are linearly independent. 
\end{rem}

Recall that $\varepsilon\in\mathrm{Aut}(\mathfrak{H})$ has been defined by $\varepsilon(x)=x+y,~\varepsilon(y)=-y$, and the anti-automorphism $\tau$ on $\mathfrak{H}$ by $\tau(x)=y,~\tau(y)=x$. Then, we have

\begin{prop}\label{prop2}
{\it Let $n$ be a positive integer. Then the following two statements, $(\mathrm{C}_n)$ and $(\mathrm{D}_n)$ hold. 
\begin{eqnarray*}
&(\mathrm{C}_n)&\quad \varepsilon\tau\phi_{n-1}^{(c)}R_x\tau\varepsilon\in\mathfrak{W}^{\prime}_n. \\
&(\mathrm{D}_n)&\quad \varepsilon\tau R_y^{-1}\partial_n^{(c)}R_y\tau\varepsilon = -\lambda(\varepsilon\tau\phi_{n-1}^{(c)}R_x\tau\varepsilon)\in \mathfrak{W}_n ~\mathit{on}~\mathfrak{H}^1.
\end{eqnarray*}
}
\end{prop}

By (\ref{17}), the expression $R_y^{-1}$ in $(\mathrm{D}_n)$ has a well-defined meaning. According to $(\mathrm{D}_n)$, there exists an element $w\in\mathfrak{H}y$ such that
\begin{equation}
\varepsilon\tau R_y^{-1}\partial_n^{(c)}R_y\tau\varepsilon = \mathcal{H}_w, \label{3}
\end{equation}
which is equivalent to Key Proposition \ref{keyprop} in $\S \ref{sec2}$ because of $R_y\tau =\tau L_x$. Therefore, Proposition \ref{prop2} is proven instead of Key Proposition \ref{keyprop}.
\begin{rem}
Here, note that $w$ in $(\ref{3})$ can be determined as follows. Equation $(\ref{3})$ holds on $\mathfrak{H}^1$, and so also hold on $\mathbb{Q}$. Since $\partial_n^{(c)}(y)\in\mathfrak{H}^0$ by Proposition $\ref{cor2}$, $R_y^{-1}\partial_n^{(c)}(y)\in x\mathfrak{H}$. Hence,
$$ \varepsilon\tau R_y^{-1}\partial_n^{(c)}R_y\tau\varepsilon (1)=\varepsilon\tau R_y^{-1}\partial_n^{(c)}(y)\in\varepsilon\tau (x\mathfrak{H})=\varepsilon (\mathfrak{H}y)=\mathfrak{H}y. $$
On the other hand, $\mathcal{H}_w(1)=w$. Therefore, by $(\ref{3})$, $w=\varepsilon\tau R_y^{-1}\partial_n^{(c)}(y)~(\in\mathfrak{H}y)$. 
\end{rem}

For the proof of Proposition \ref{prop2}, following lemmata are needed.

\begin{lem}\label{lem2}
{\it For any $X\in\mathfrak{W}^{\prime}$ and any $l\ge 1$, we have $[\lambda(X),L_{z_l}]=XL_{z_l}+L_{x^l}X$. }
\end{lem}

\begin{proof}
It is sufficient to show the case in which $X=L_{z_k}\mathcal{H}_w$, which follows directly from 
\begin{equation}
[\mathcal{H}_{z_kw},L_{z_l}]=L_{z_k}\mathcal{H}_wL_{z_l}+L_{z_{k+l}}\mathcal{H}_w, \label{16}
\end{equation}
the harmonic product rule.
\end{proof}

\begin{lem}\label{lem3}
{\it For any $k,l\ge 1$, we have $(\lambda -1)(\mathfrak{W}^{\prime}_k)L_{z_l}\subset \mathfrak{W}^{\prime}_{k+l}$. }
\end{lem}

\begin{proof}
The proof follows directly from (\ref{16}). 
\end{proof}

\begin{lem}\label{lem4}
{\it We have $(\lambda -1)(\mathfrak{W}^{\prime}_k)\cdot(\lambda -1)(\mathfrak{W}^{\prime}_l)\subset(\lambda -1)(\mathfrak{W}^{\prime}_{k+l})$ for any $k,l\ge 1$. }
\end{lem}

\begin{proof} 
Let $d$ and $d^{\prime}$ be the weights of words $w$ and $w^{\prime}$, respectively. The assertion $(\lambda -1)(L_{z_k}\mathcal{H}_w)\cdot(\lambda -1)(L_{z_l}\mathcal{H}_{w^{\prime}})\in (\lambda -1)(\mathfrak{W}^{\prime}_{k+l+d+d^{\prime}})$ is only necessary to show.
\begin{eqnarray*}
\text{LHS} &=& (\mathcal{H}_{z_kw}-L_{z_k}\mathcal{H}_w)(\mathcal{H}_{z_lw^{\prime}}-L_{z_l}\mathcal{H}_{w^{\prime}}) \\
&=& \mathcal{H}_{z_kw\ast z_lw^{\prime}}-\mathcal{H}_{z_kw}L_{z_l}\mathcal{H}_{w^{\prime}}-L_{z_k}\mathcal{H}_{w\ast z_lw^{\prime}}+L_{z_k}\mathcal{H}_wL_{z_l}\mathcal{H}_{w^{\prime}} \\
&=& \mathcal{H}_{z_k(w\ast z_lw^{\prime})+z_l(z_kw\ast w^{\prime})+z_{k+l}(w\ast w^{\prime})}-(L_{z_k}\mathcal{H}_wL_{z_l} \\
&{}& \quad +L_{z_l}\mathcal{H}_{z_kw}+L_{z_{k+l}}\mathcal{H}_w)\mathcal{H}_{w^{\prime}}-L_{z_k}\mathcal{H}_{w\ast z_lw^{\prime}}+L_{z_k}\mathcal{H}_wL_{z_l}\mathcal{H}_{w^{\prime}} \\
&=& \mathcal{H}_{z_k(w\ast z_lw^{\prime})}-L_{z_k}\mathcal{H}_{w\ast z_lw^{\prime}}+\mathcal{H}_{z_l(z_kw\ast w^{\prime})}-L_{z_l}\mathcal{H}_{z_kw\ast w^{\prime}} \\
&{}& \quad +\mathcal{H}_{z_{k+l}(w\ast w^{\prime})}-L_{z_{k+l}}\mathcal{H}_{w\ast w^{\prime}} \\
&=& (\lambda -1)(L_{z_k}\mathcal{H}_{w\ast z_lw^{\prime}}+L_{z_l}\mathcal{H}_{z_kw\ast w^{\prime}}+L_{z_{k+l}}\mathcal{H}_{w\ast w^{\prime}}). \\
&\in& \mathrm{RHS}.
\end{eqnarray*}
Hence, the lemma is proven. 
\end{proof}

\begin{lem}\label{lem5}
{\it For any $X\in\mathfrak{W}^{\prime}$, we have $\lambda(X)(1)=X(1)$.}
\end{lem}

\begin{proof}
$(\lambda-1)(L_{z_k}\mathcal{H}_w)(1)=\mathcal{H}_{z_kw}(1)-L_{z_k}\mathcal{H}_w(1)=z_kw-z_kw=0.$
\end{proof}

\begin{lem}\label{lem6}
{\it Let $X\in\mathfrak{W}$. If $X(1)=0$ and $[X,L_{z_k}]=0$ for any $k\ge 1$, we have $X=0$.}
\end{lem}

\begin{proof}
If $[X,L_{z_k}]=0$ for any $k\ge 1$,
$$ X(z_{k_1}\cdots z_{k_n})=z_{k_1}X(z_{k_2}\cdots z_{k_n}) = \cdots = z_{k_1}\cdots z_{k_n}X(1)=0. $$
\end{proof}

Using their validity and various properties obtained in the proof of Proposition \ref{comm}, Proposition \ref{prop2} can be shown as follows. 

\begin{proof}
In the following, ($\mathrm{C}_n$) and ($\mathrm{D}_n$) are proven inductively as $(\mathrm{C}_1)\Rightarrow(\mathrm{D}_1)\Rightarrow(\mathrm{C}_2)\Rightarrow(\mathrm{D}_2)\Rightarrow(\mathrm{C}_3)\Rightarrow\cdots$. 

Since $\varepsilon\tau\phi_0^{(c)}R_x\tau\varepsilon = -L_y\in\mathfrak{W}_1^{\prime}$, the claim $(\mathrm{C}_1)$ holds. 

Assume that $(\mathrm{C}_n)$ is proven. Note that we have the equality
\begin{equation}
R_y^{-1}\partial_n^{(c)}R_y=\partial_n^{(c)}-\phi_{n-1}^{(c)}R_x \label{14}
\end{equation}
based on $(\mathrm{A}_n)$, Lemma \ref{lem1}, and Proposition \ref{cor2}. Then, we obtain
\begin{eqnarray*}
&{}& [\varepsilon\tau R_y^{-1}\partial_n^{(c)}R_y\tau\varepsilon, L_{z_k}] = \varepsilon\tau R_y^{-1}\partial_n^{(c)}R_y\tau\varepsilon L_{z_k} - L_{z_k}\varepsilon\tau R_y^{-1}\partial_n^{(c)}R_y\tau\varepsilon \\
&{}& =\varepsilon\tau\partial_n^{(c)}\tau\varepsilon L_{z_k} -\varepsilon\tau\phi_{n-1}^{(c)}R_x\tau\varepsilon L_{z_k} - L_{z_k}\varepsilon\tau\partial_n^{(c)}\tau\varepsilon +L_{z_k}\varepsilon\tau\phi_{n-1}^{(c)}R_x\tau\varepsilon . \label{hosi}
\end{eqnarray*}
Note that
\begin{equation}
\varepsilon L_x=L_z\varepsilon ,~\varepsilon L_y=-L_y\varepsilon ,~\tau L_x=R_y\tau ,~\tau L_y=R_x\tau . \label{15}
\end{equation}
Using (\ref{15}), the first term of the expression (\ref{hosi}) turns into $-\varepsilon\tau\partial_n^{(c)}R_{z^{k-1}}R_x\tau\varepsilon$. According to $(\mathrm{A}_n)$, $(\alpha_n)$, and Lemma \ref{lem1}, 
$$ -\varepsilon\tau\partial_n^{(c)}R_{z^{k-1}}R_x\tau\varepsilon = -\varepsilon\tau R_{z^{k-1}}(R_x\partial_n^{(c)}+R_y\phi_{n-1}^{(c)}R_x)\tau\varepsilon . $$
Again apply (\ref{15}). Then, two terms cancel and two others combine to the second term on the right in the statement below it. 
$$ [\varepsilon\tau R_y^{-1}\partial_n^{(c)}R_y\tau\varepsilon, L_{z_k}] = -\varepsilon\tau\phi_{n-1}^{(c)}R_x\tau\varepsilon L_{z_k}-L_{x^k}\varepsilon\tau\phi_{n-1}^{(c)}R_x\tau\varepsilon. $$
This is equal to $[\lambda(-\varepsilon\tau\phi_{n-1}^{(c)}R_x\tau\varepsilon),L_{z_k}]$ by Lemma \ref{lem2} and $(\mathrm{C}_n)$. Moreover,
$$ \varepsilon\tau R_y^{-1}\partial_n^{(c)}R_y\tau\varepsilon (1) = \varepsilon (\tau\partial_n^{(c)}\tau -\tau\phi_{n-1}^{(c)}R_x\tau)\varepsilon (1) = -\varepsilon\tau \phi_{n-1}^{(c)}R_x\tau\varepsilon (1) $$
owing to (\ref{14}) and $\partial_n^{(c)}(1)=0$. By Lemma \ref{lem5}, this equals $-\lambda(\varepsilon\tau \phi_{n-1}^{(c)}R_x\tau\varepsilon) (1)$. Hence, by Lemma \ref{lem6}, we have $(\mathrm{D}_n)$: $\varepsilon\tau R_y^{-1}\partial_n^{(c)}R_y\tau\varepsilon = -\lambda(\varepsilon\tau\phi_{n-1}^{(c)}R_x\tau\varepsilon) $ on $\mathfrak{H}^1$.

Next, assume that $(\mathrm{D}_n)$ is proven. Using (\ref{14}) and $(\mathrm{D}_n)$, we obtain
$$ \varepsilon\tau \partial_n^{(c)}\tau\varepsilon = \varepsilon\tau R_y^{-1}\partial_n^{(c)}R_y\tau\varepsilon +\varepsilon\tau\phi_{n-1}^{(c)}R_x\tau\varepsilon=(\lambda -1)(-\varepsilon\tau\phi_{n-1}^{(c)}R_x\tau\varepsilon). $$
According to $(\mathrm{B}_n)$, we have the expression
$$ \phi_n^{(c)}=\sum_{i=0}^nf_i^{(c)}R_{z^{n-i}}~\quad~(f_i^{(c)}\in\mathbb{Q}[\partial_1^{(c)},\ldots,\partial_i^{(c)}]_{(i)}). $$
Hence,
$$ \displaystyle \varepsilon\tau\phi_n^{(c)}R_x\tau\varepsilon = \varepsilon \tau\sum_{i=0}^nf_i^{(c)}R_{z^{n-i}} R_x\tau\varepsilon = -\sum_{i=0}^n\varepsilon\tau f_i^{(c)}\tau\varepsilon L_{z_{n+1-i}}. $$
By Lemma \ref{lem4}, this is an element of $\sum_{i=0}^n(\lambda -1)(\mathfrak{W}^{\prime}_i) L_{z_{n+1-i}}$. Then, by Lemma \ref{lem3}, this is a subset of $\mathfrak{W}_{n+1}^{\prime}$. Hence, $(\mathrm{C}_{n+1})$ is proven. 
\end{proof}

\section{Alternative Extension of $\partial_n$}\label{sec6}

\noindent Here, an alternative operator $\widehat{\partial}_n^{(c)}$ is defined instead of $\partial_n^{(c)}$ in Definition \ref{def1}. In this section, several properties of $\widehat{\partial}_n^{(c)}$'s are discussed. In particular, $\partial_n^{(c)}$ and $\widehat{\partial}_n^{(c)}$ give the same class of relations for MZV's. 

\begin{defn}\label{defn3}
{\it Let $c$ be a rational number and $H$ the same operator as in Definition $\ref{def1}$. For each integer $n \ge 1$, the $\mathbb{Q}$-linear map $\widehat{\partial}_n^{(c)}:\mathfrak{H}\to\mathfrak{H}$ is defined by}
$$ \displaystyle \widehat{\partial}_n^{(c)}=\frac{1}{(n-1)!}\mathrm{ad}(\widehat{\theta}^{(c)})^{n-1}(\partial_1) $$
{\it where $\widehat{\theta}^{(c)}$ is the $\mathbb{Q}$-linear map defined by $\widehat{\theta}^{(c)}(x)=\theta (x)$, $\widehat{\theta}^{(c)}(y)=\theta (y)$ and the rule}
\begin{equation}
\widehat{\theta}^{(c)}(ww^{\prime})=\widehat{\theta}^{(c)}(w)w^{\prime}+w\widehat{\theta}^{(c)}(w^{\prime})+cH(w)\partial_1(w^{\prime}) \label{18}
\end{equation}
{\it for any $w,w^{\prime} \in \mathfrak{H}$.}
\end{defn}

The only difference between $\theta^{(c)}$ and $\widehat{\theta}^{(c)}$ is the order of $H$ and $\partial_1$ appearing in the right-hand side of (\ref{19}) and (\ref{18}). 

\begin{lem}\label{lem7}
{\it For any rational number $c$, we have $\widehat{\theta}^{(c)}=\theta^{(-c)}+c\partial_1(H-1)$. }
\end{lem}

\begin{proof}
Calculate the recursive rules for both sides. 
\end{proof}

\begin{prop}\label{prop3}
{\it For any positive integer $n$ and any rational number $c$, we have $\widehat{\partial}_n^{(c)}\in \mathbb{Q}[\partial_1^{(-c)},\ldots ,\partial_n^{(-c)}]$. }
\end{prop}

\begin{proof}
The proposition holds for $n=1$ because $\widehat{\partial}_1^{(c)}=\partial_1^{(-c)}=\partial_1$. Assume that the proposition is proven for $n$. Using Lemma \ref{lem7}, we obtain
$$ n\widehat{\partial}_{n+1}^{(c)} = [\widehat{\theta}^{(c)},\widehat{\partial}_{n+1}^{(c)}] = [\theta^{(-c)}+c\partial_1(H-1),\widehat{\partial}_n^{(c)}] = [\theta^{(-c)},\widehat{\partial}_n^{(c)}]+c(n-1)\partial_1\widehat{\partial}_n^{(c)}. $$
Hence, by induction, the proposition holds for $n+1$. 
\end{proof}

\begin{exmp}\label{ex2}
The polynomials in Proposition $\ref{prop3}$ can be constructed explicitly. For example, 
\begin{eqnarray*}
\widehat{\partial}_2^{(c)} &=& \partial_2^{(-c)}+c\partial_1^2, \\
\widehat{\partial}_3^{(c)} &=& \partial_3^{(-c)}+2c\partial_1\partial_2^{(-c)}+c^2\partial_1^3, \\
\widehat{\partial}_4^{(c)} &=& \partial_4^{(-c)}+\frac{7}{3}c\partial_1\partial_3^{(-c)}+\frac{2}{3}c{\partial_2^{(-c)}}^2+3c^2\partial_1^2\partial_2^{(-c)}+c^3\partial_1^4.
\end{eqnarray*}
\end{exmp}

\begin{cor}\label{cor3}
{\it For any rational numbers $c, c^{\prime}$, and any positive integers $n, m$, we have $[\partial_n^{(c)},\widehat{\partial}_m^{(c^{\prime})}]=0$. }
\end{cor}

\begin{proof}
The proof follows immediately from Proposition \ref{comm} and \ref{prop3}. 
\end{proof}

\begin{lem}\label{lem8}
{\it For any rational number $c$, we have $\widehat{\theta}^{(c)}=\tau \theta^{(-c)}\tau $. }
\end{lem}

\begin{proof}
By direct calculations, each image of $x$ and $y$ of $\mathfrak{H}$ coincides. Write $w=w_1w_2$ where $w_1$ and $w_2$ are words of $\mathfrak{H}$ with $\deg (w_i)\ge 1$, $i=1,2$. Then,
\begin{eqnarray*}
\tau \theta^{(-c)} \tau (w) &=& \tau \theta^{(-c)}\bigl(\tau (w_2)\tau (w_1)\bigl) \\
&=& \tau \bigl(\theta^{(-c)}\tau (w_2)\tau (w_1)+\tau (w_2)\theta^{(-c)}\tau (w_1)-c\partial_1\tau (w_2)H\tau (w_1)\bigl) \\
&=& w_1\tau \theta^{(-c)} \tau (w_2)+\tau \theta^{(-c)} \tau (w_1)w_2 -c \tau H\tau (w_1)\tau \partial_1 \tau (w_2).
\end{eqnarray*}
Use $\tau H\tau =H,~\tau \partial_1 \tau =-\partial_1$ to complete the proof. 
\end{proof}

\begin{prop}\label{prop4}
{\it For any positive integer $n$ and any rational number $c$, we have $\widehat{\partial}_n^{(c)}=- \tau \partial_n^{(-c)}\tau $. }
\end{prop}

\begin{proof}
The proof is given by induction on $n$. The proposition holds for $n=1$. Assume that the proposition is proven for $n$. Using Lemma \ref{lem8}, we have
$$ (n+1)\widehat{\partial}_{n+1}^{(c)} = [\widehat{\theta}^{(c)},\widehat{\partial}_n^{(c)}] = -[\tau \theta^{(-c)} \tau, \tau \partial_n^{(-c)} \tau] = -\tau [\theta^{(-c)},\partial_n^{(-c)}]\tau = -n \tau \partial_{n+1}^{(-c)}\tau . $$
Thus, the proposition holds for $n+1$.
\end{proof}

By Proposition \ref{prop4}, we have $\widehat{\partial}_n^{(c)}(\mathfrak{H}^0)\subset\mathrm{ker}\mathit{Z}$, which assigns the same class to Theorem because of Proposition \ref{prop3}.

\section*{Appendix 1: A New Proof of Derivation Relation}\label{sec5}

\noindent In the case of $c=0$ in Theorem \ref{mthm}, we have an alternative proof of the derivation relation for MZV's, reducing to Kawashima's relation. Here, the automorphisms on $\widehat{\mathfrak{H}}$, the completion of $\mathfrak{H}$, are introduced. (See \cite{IKZ} for details.) Let $\Phi$ be the automorphism on $\widehat{\mathfrak{H}}$ defined by $\Phi(x)=x$ and $\Phi(z)=z(1+y)^{-1}$. The automorphism $\Phi$ satisfies
$$ \displaystyle \frac{1}{1+y}\ast w=\frac{1}{1+y}\Phi(w) $$
for $w\in\mathfrak{H}^1$ (\cite[Proposition 6]{IKZ}). Let $\Delta$ be $\exp(\sum_{n\ge 1}\frac{\partial_n}{n})$ which is the automorphism on $\widehat{\mathfrak{H}}$ characterized by $\Delta(x)=x(1-y)^{-1}$ and $\Delta(z)=z$. Then, we have $\Phi=\varepsilon\Delta\varepsilon $ on $\widehat{\mathfrak{H}}$. This implies that $\mathcal{H}_{\frac{1}{1+y}}=\varepsilon L_x^{-1}\Delta L_x\varepsilon $ on $\widehat{\mathfrak{H}}^1$, the completion of $\mathfrak{H}^1$. Hence, $(\Delta -1)(\mathfrak{H}^0)\subset L_x\varepsilon(\mathfrak{H}y\ast\mathfrak{H}y)$. Expanding the exponential map, each degree $i$ part of $\Delta -1$ sends $\mathfrak{H}^0$ to $L_x\varepsilon(\mathfrak{H}y\ast\mathfrak{H}y)$, and, therefore, the derivation relation is a class of relations of MZV's according to Kawashima's relation in Fact \ref{thm2}.

\section*{Appendix 2: Ohno's relation and Derivation Relation}

For $n\ge 1$, the derivation $D_n$ on $\mathfrak{H}$ is defined by $D_n(x)=0$, $D_n(y)=x^ny$. The map $\bar{D}_n=\tau D_n\tau$ is another derivation on $\mathfrak{H}$ such that $\bar{D}_n(x)=xy^n$, $\bar{D}_n(y)=0$. Set
$$ \displaystyle \sigma=\sum_{l=0}^{\infty}\sigma_l=\exp\biggl(\sum_{n=1}^{\infty}\frac{D_n}{n}\biggr),~\bar{\sigma}=\sum_{l=0}^{\infty}\bar{\sigma}_l=\exp\biggl(\sum_{n=1}^{\infty}\frac{\bar{D}_n}{n}\biggr). $$
The maps $\sigma$, $\bar{\sigma}$ are automorphisms on $\mathfrak{H}$. Putting $D=\sum_{n=1}^{\infty}\frac{D_n}{n}$, we find $D^m(x)=0$, $D^m(y)=(-\log(1-x))^my$ for $m\ge 1$, and hence, 
$$ \sigma(x)=x,~\sigma(y)=\frac{1}{1-x}y. $$
Since the map $\sigma$ is an automorphism, 
\begin{eqnarray*}
\displaystyle \sigma(x^{k_1-1}y\cdots x^{k_n-1}y) &=& x^{k_1-1}\frac{1}{1-x}y\cdots x^{k_n-1}\frac{1}{1-x}y \\
\displaystyle &=& \sum_{l=0}^{\infty}\sum_{\begin{subarray}{c}e_1+\cdots +e_n=l, \\ e_1,\ldots,e_n\ge 0 \end{subarray}}x^{k_1+e_1-1}y\cdots x^{k_n+e_n-1}y,
\end{eqnarray*}
and hence, 
$$ \displaystyle \sigma_l(x^{k_1-1}y\cdots x^{k_n-1}y)=\sum_{\begin{subarray}{c}e_1+\cdots +e_n=l, \\ e_1,\ldots,e_n\ge 0 \end{subarray}}x^{k_1+e_1-1}y\cdots x^{k_n+e_n-1}y. $$
Thus, Ohno's relation can be stated as $\sigma_l(1-\tau)(\mathfrak{H}^0)\subset\mathrm{ker}\mathit{Z}$ for any $l\ge 0$. If $l=0$, Ohno's relation is reduced to the duality formula. 

The automorphisms $\sigma$, $\bar{\sigma}$ and $\Delta$, which has been defined in Appendix 1, have a property as follows. (See \cite[Theorem 4,(ii)]{IKZ}.) 
\begin{prop}\label{appen}
$\Delta=\bar{\sigma}\sigma^{-1}$. 
\end{prop}
According to this proposition, we have $\sigma-\bar{\sigma}=(1-\Delta)\sigma$. Since $\bar{\sigma}_l=\tau\sigma_l\tau$ and the duality formula is included in Ohno's relation, this identity implies that Ohno's relation is equivalent to the union of the duality formula and the derivation relation.



\end{document}